\documentclass[11pt, a4paper]{article}
\usepackage{amsfonts,amssymb,amsmath,amscd,latexsym,makeidx,theorem}
\usepackage{hyperref}

\newtheorem{thm}{Theorem}[section]
\newtheorem{lem}[thm]{Lemma}
\newtheorem{prop}[thm]{Proposition}

\newtheorem{defn}{Definition}[section]

\newenvironment{proof}{\noindent\emph{Proof.}}{\hfill$\square$\medskip}

\newcommand{\rn}{\mathbb{R}^{n}}

\newcommand{\s}{\mathcal{S}}

\DeclareMathOperator{\supp}{supp}
\DeclareMathOperator{\dist}{dist}
\DeclareMathOperator{\ex}{exp}

\title{Structure of conformal metrics on $\rn$ with constant Q-curvature}

\author{Ali Hyder \thanks{The author is supported by the Swiss National Science Foundation.} 
\\ {\small Universit\"at Basel}\\ {\small \texttt{ali.hyder@unibas.ch}} }

\begin{document}

\maketitle

\begin{abstract}
In this article we study the nonlocal equation 
\begin{align}
 (-\Delta)^{\frac{n}{2}}u=(n-1)!e^{nu}\quad \text{in $\rn$}, \quad\int_{\rn}e^{nu}dx<\infty, \notag
\end{align}
which arises in the conformal geometry. Inspired by the previous work of C. S. Lin and L. Martinazzi in even dimension and  T. Jin, A. Maalaoui, L. Martinazzi, J. Xiong
in dimension three we classify all solutions to the above equation in terms of their behavior at infinity.
\end{abstract}

\section{Introduction to the problem and the main theorems}
In this paper we consider the equation 
\begin{align} \label{main-eq}
 (-\Delta)^{\frac{n}{2}}u=(n-1)!e^{nu}\quad \text{in $\rn$}. 
\end{align}
Here we assume that
\begin{align}\label{V}
 V:=\int_{\rn}e^{nu}dx<\infty,
\end{align}
and we shall see both the left and right-hand side of \eqref{main-eq} as tempered distributions. In order to define the left-hand side of \eqref{main-eq}
as a tempered distribution, one possibility is to follow the approach of \cite{TALJ}, i.e. we see the operator $(-\Delta)^{\frac{n}{2}}$ as 
$(-\Delta)^{\frac{n}{2}}:=(-\Delta)^{\frac{1}{2}}\circ(-\Delta)^{\frac{n-1}{2}}$ for $n\geq 1$ odd integer with the convention that $(-\Delta)^0$ is the identity, 
where $(-\Delta)^\frac{1}{2}$ is defined as follows.
First for $s>0$ consider the space
\begin{equation}\label{Lsigma}
L_{s}(\rn):=\left\{v\in L^1_{loc}(\rn):\int_{\rn}\frac{|v(x)|}{1+|x|^{n+2s}}dx<\infty\right\}.
\end{equation}
Then for $v\in L_s (\rn)$ we define $(-\Delta)^s v$ as the tempered distribution defied by
\begin{equation}\label{lapweak}
\langle(-\Delta)^s v, \varphi\rangle:=\int_{\rn}v(-\Delta)^s\varphi dx \qquad \text{for every }\varphi\in\mathcal{S}(\rn),
\end{equation}
where
$$\mathcal{S}(\rn):=\left\{u\in C^{\infty}(\rn): \sup_{x\in\rn}|x|^N|D^{\alpha}u(x)|<\infty \text{ for all } N\in\mathbb{N} \text{ and } \alpha\in
 \mathbb{N}^n\right\}$$
 is the Schwartz space, and
 $$\widehat{(-\Delta)^s \varphi }(\xi)=|\xi|^{2s}\hat{\varphi}(\xi),\quad \text{for }\varphi\in \s(\rn).$$
 Here the normalized Fourier transform is defined by 
 $$\mathcal{F}(f)(\xi):=\hat{f}(\xi):=\frac{1}{(2\pi)^{n/2}}\int_{\rn}f(x)e^{-ix\cdotp\xi}dx, \quad f\in L^1(\rn).$$
 Notice that the integral in \eqref{lapweak} converges thanks to Proposition \ref{ests} below.

 Then a possible definition of the equation
 \begin{equation}\label{eqf}
 (-\Delta)^{\frac{n}{2}}u=f\quad \text{in }\rn
 \end{equation}
 is the following:
 
\begin{defn}\label{def-soln}
Given $f\in\s'(\rn)$, we say that $u$ is a solution of \eqref{eqf} if
$$u\in W^{n-1,1}_{loc}(\rn),\quad \Delta^{\frac{n-1}{2}}u\in L_{\frac{1}{2}}(\rn),$$
and
\begin{align}\label{k0}
\int_{\rn}(-\Delta)^{\frac{n-1}{2}}u(x)(-\Delta)^{\frac{1}{2}}\varphi(x)dx=\langle f, \varphi\rangle,\quad \text{for every }\varphi\in\mathcal{S}(\rn).
\end{align}
\end{defn}

While Definition \ref{def-soln} is general enough for our purposes, requiring a priori that a solution to \eqref{main-eq} belongs to $W^{n-1,1}_{loc}(\rn)$ might sound unnecessarily restrictive.
In fact it is possible to relax Definition \ref{def-soln} as follows.

\begin{defn}\label{def-new}
Given $f\in\s'(\rn)$, a function $u\in L_{\frac{n}{2}}(\rn)$ is a solution of \eqref{eqf} if
\begin{align}\label{k}
\int_{\rn}u(x)(-\Delta)^{\frac{n}{2}}\varphi(x)dx=\langle f, \varphi\rangle,\quad \text{for every }\varphi\in\mathcal{S}(\rn).
\end{align}
\end{defn}

Notice again that the integral in \eqref{k0} and \eqref{k} are converging by Proposition \ref{ests} below.

\medskip

As we shall see, a function $u$ solving \eqref{main-eq}-\eqref{V} in the sense of Definition \ref{def-new} also solves \eqref{main-eq} in the sense of Definition \ref{def-soln}, and conversely, see Proposition \ref{defequiv} below. 
Therefore, from now on a solution of \eqref{main-eq}-\eqref{V} will be intended in the sense of Definition \ref{def-soln}.
 In fact it turns out that such solutions enjoy even more regularity:


 \begin{thm}\label{regularity}
 Let $u$ be a solution of \eqref{main-eq}-\eqref{V} (in the sense of Definition \ref{def-soln} or \ref{def-new}). Then $u$ is smooth. 
 \end{thm}
 
Geometrically any solution $u$ of \eqref{main-eq}-\eqref{V} corresponds to a conformal metric $g_u:=e^{2u}|dx|^2$ on $\rn$ ($|dx|^2$ is the 
Euclidean metric on $\rn$) such that the $Q$-curvature of $g_u$ is constant $(n-1)!$. Moreover the volume and the total $Q$-curvature of the metric $g_u$ are
$V=\int_{\rn}e^{nu}dx<\infty$ and $\int_{\rn}(n-1)!e^{nu}dx<\infty$ respectively. When $n=1$ a geometric interpretation of \eqref{main-eq} in terms 
of holomorphic immersion of $\overline{D^2}$ into $\mathbb{C}$ was given in [\cite{DMR}, Theorem 1.3]. If $u$ is a solution of \eqref{main-eq}
then for any constant $c$, $\tilde{u}:=u-c$ satisfies $$(-\Delta)^{\frac{n}{2}}\tilde{u}=(n-1)!e^{nc}e^{n\tilde{u}}\quad \text{in $\rn$}.$$ This shows that 
we could take any arbitrary positive constant instead of $(n-1)!$ in \eqref{main-eq}, but we restrict ourselves  to the fixed constant $(n-1)!$ because
it is the constant $Q$-curvature of the round sphere $S^n$. 

Now we shall address the following question: What are the solutions to \eqref{main-eq} and in particular how do they behave at infinity? 

It is well known that the equation \eqref{main-eq} possess the following explicit solution
 $$u(x)=\log\left(\frac{2}{1+|x|^2}\right),$$ obtained by pulling back the round metric on $S^n$ via the stereographic projection.


By translating and rescaling this function $u$ one can produce a class of solutions, namely  
$$u_{\lambda,x_0}(x):=\log\left(\frac{2\lambda}{1+\lambda^2|x-x_0|^2}\right),$$ for every $\lambda>0$ and $x_0\in\rn$. 
Any such $u_{\lambda,x_0}$ is called spherical solution. W. Chen-C. Li \cite{Chen-Li} showed that these are the only solutions in dimension two but in higher dimension  nonspherical 
solutions do exist as shown by A. Chang-W. Chen \cite{CC}. C. S. Lin \cite{Lin} for $n=4$ and L. Martinazzi \cite{LM} for $n\geq 4$ even classified all solutions of \eqref{main-eq}-\eqref{V}
and they proved: 

\medskip
\noindent \textbf{Theorem A (\cite{Lin}, \cite{LM})}
\emph{Any solution $u$ of \eqref{main-eq}-\eqref{V} with $n$ even has the asymptotic behavior 
\begin{align}\label{asymp}
u(x)=-P(x)-\alpha\log|x| +o(\log|x|)
\end{align}
where $\alpha=\frac{2V}{|S^n|}$, $\frac{o\left(\log|x|\right)}{\log|x|}\to 0$ as $|x|\to\infty$ and
$P$ is a polynomial bounded from below and of degree at most $n-2$.}  

\medskip
A partial converse of Theorem A holds true.
For a given  $0<\alpha<2$ and a given polynomial $P$ such that degree$(P)\leq n-2$ and
$x\cdotp \nabla P(x)\to\infty$ as $|x|\to\infty$, J. Wei-D. Ye \cite{W-Y} in dimension four and A. Hyder-L. Martinazzi \cite{H-M} in even dimension $n\geq 4$  proved the existence of solutions
of  \eqref{main-eq}-\eqref{V} with asymptotic behavior given in \eqref{asymp}. 

When $n$ is odd things are more complex as the operator $(-\Delta)^{\frac{n}{2}}$ is nonlocal. In a recent work T. Jin, A. Maalaoui, L. Martinazzi, J. Xiong have proven the following theorem in 
dimension three:

\medskip
\noindent \textbf{Theorem B (\cite{TALJ})}\emph{ Let $u$ be a smooth solution of \eqref{main-eq}-\eqref{V} with $n=3$. Then $u$ has the asymptotic behavior
given by \eqref{asymp}, 
where $P$ is a polynomial of degree $0$ or $2$ bounded from below, $\alpha\in(0,2]$ and $\alpha=2$ if and only if degree$(P)=0$. Moreover for every  $0<\alpha<2$ there exist at least 
one smooth solution of \eqref{main-eq}-\eqref{V}.} 
\medskip

In analogy with Theorem A and B we study the asymptotic behavior of smooth solutions to the problem \eqref{main-eq}-\eqref{V} in odd dimension. In order to do that we define
\begin{align}\label{def-v}
 v(x):=\frac{(n-1)!}{\gamma_n}\int_{\rn}\log\left(\frac{1+|y|}{|x-y|}\right)e^{nu(y)}dy, \quad \gamma_n=\frac{(n-1)!}{2}|S^n|,
\end{align}
where $u$ is a smooth solution of \eqref{main-eq}-\eqref{V} and we prove

\begin{thm}\label{thm-1}
Let $n\geq 3$ be any odd integer and let $u$ be a smooth solution of \eqref{main-eq}-\eqref{V}. Then $$u=v+P,$$ where  
 $P$ is a polynomial of degree at most $n-1$ bounded from above, $v$ is given by \eqref{def-v} and it satisfies 
 $$v(x)=-\alpha\log|x|+o(\log|x|),\quad as\,|x|\to\infty,$$
 where $\alpha=\frac{2V}{|S^n|}$. Moreover 
 $$\lim_{|x|\to\infty}D^{\beta}v(x)=0\, \text{ for every multi-index $\beta\in\mathbb{N}^n$ with }0<|\beta|\leq n-1.$$
\end{thm}

Under certain assumptions on the polynomial $P$, a partial converse of Theorem \ref{thm-1} has been proven by A. Hyder \cite{A-H}, namely 

\medskip
\noindent\textbf{Theorem C (\cite{A-H})}
\emph{Let $n\geq 3$ be an odd integer. For any given $V\in(0, |S^n|)$ and any given polynomial $P$ of degree at most $n-1$ such that
\begin{equation}\label{condP2}
P(x)\to\infty \quad \text{as } |x|\to\infty,
\end{equation}
there exists $u\in C^{\infty}(\rn)\cap L_{\frac n2}(\rn)$ solution of \eqref{main-eq}-\eqref{V} having the asymptotic behavior given in \eqref{asymp} with $\alpha=\frac{2V}{|S^n|}$.}
\medskip

Using Theorem \ref{thm-1} one can obtain necessary and sufficient conditions under which any solution of \eqref{main-eq}-\eqref{V} is spherical. More precisely we have the 
following theorem.
\begin{thm}\label{thm-1b}
 Let $u$ be a smooth solution of \eqref{main-eq}-\eqref{V}. Then the following are equivalent:
\begin{itemize}
 \item [(i)] $u$ is a spherical solution.
 \item[(ii)] $deg(P)=0$, where $P$ is the polynomial given by Theorem \ref{thm-1}.
 \item[(iii)] $u(x)=o(|x|^2)$ as $|x|\to\infty.$
 \item[(iv)] $\lim_{|x|\to\infty}\Delta^ju(x)=0$ for $j=1,2,...,\frac{n-1}{2}$.
 \item[(v)] $\lim\inf_{|x|\to\infty}$ $R_{g_u}>-\infty$, where $R_{g_u}$ is the scalar curvature of $g_u$.
 \item[(vi)] $\pi^*g_u$ can be extended to a Riemannian metric on $S^n$, where $\pi$ is the stereographic projection.
\end{itemize}
Moreover, if $u$ is not a spherical solution then there exists a $j$ with $1\leq j\leq\frac{n-1}{2}$ and a constant $c<0$ such that
\begin{align}\label{C}
 \lim_{|x|\to\infty}\Delta^ju(x)=c.
\end{align}
\end{thm}

The equivalence $(i)\Leftrightarrow (vi)$ was proven by Chang-Yang \cite{C-Y} for $n\geq 3$ odd or even using moving plane technique.

In dimension 3 and 4 if $u$ is a smooth solution of \eqref{main-eq}-\eqref{V} then $V\in(0, |S^n|]$ (see \cite{Lin}, \cite{TALJ}) but $V$ could be greater than $|S^n|$ in
higher dimension. For instance in dimension 6, L. Martinazzi \cite{LM-vol} proved the existence of solution with large volume. In a recent work  X. Huang-D. Ye \cite{HD}  in dimension $n=4k+2$ with $k\geq 1$
have shown the existence of solution for any volume $V\in(0,\infty)$. What would be the precise range of the volume $V$ in 
dimension $n\geq 5$ odd or $n$ is of the form $n=4k$ and $k\geq 2$ is an open question.

We also mention that using different techniques F. Da Lio, L. Martinazzi and T. Rivi\`ere \cite{DMR} have discussed the case in one dimension, proving that
all solutions are spherical. 
\section{Definitions, regularity issues and proof of Theorem \ref{regularity}}

\begin{prop}\label{ests} For any $s>0$ and $\varphi\in \s(\rn)$ we have
$$|(-\Delta)^s \varphi(x)|\le \frac{C}{|x|^{n+2s}},$$
where $(-\Delta)^s\varphi :=(-\Delta)^\sigma\circ (-\Delta)^k\varphi$, where $\sigma\in [0,1)$, $k\in \mathbb{N}$ and $s=k+\sigma$.
\end{prop}

In order to prove Proposition \ref{ests} let us introduce the spaces
\begin{align}
 \s_k(\rn):&=\left\{\varphi\in\s(\rn): D^{\alpha}\hat{\varphi}(0)=0, \text{ for } |\alpha|\leq k\right\} \notag\\
      & =\left\{\varphi\in\s(\rn): \int_{\rn}y^{\alpha}\varphi(y)dy=0, \text{ for } |\alpha|\leq k\right\},\quad k=0,1,2,\dots\notag\\
    \s_{-1}(\rn): &= \s(\rn)\notag
\end{align}
Proposition \ref{ests} easily follows from the remark that $\Delta^k \varphi\in \s_{2k-1}(\rn)$ for $k\in \mathbb{N}$ and $\varphi\in\s(\rn)$, and from Lemma \ref{S_k} below.

\begin{lem}\label{S_k}
 Let $\varphi\in\s_k(\rn)$ and $\sigma\in (0,1)$. Then   $$|(-\Delta)^{\sigma}\varphi(x)|\leq \frac{C}{|x|^{n+2\sigma+k+1}},\quad x\in\rn.$$
\end{lem}
\begin{proof}
Since $(-\Delta)^{\sigma}\varphi\in C^{\infty}(\rn)$ for $\varphi \in\s(\rn)$, it suffices to prove the lemma for large $x$. For a fix $x\in\rn$ we split $\rn$ into 
$$A_1:=B_\frac{|x|}{2}\quad \text{and } A_2:=\rn\setminus B_\frac{|x|}{2}.$$ Then using \eqref{b} we have
 \begin{align}
 \left|(-\Delta)^{\sigma}\varphi(x)\right|&\leq \frac{1}{2}C_{n,\sigma}\left(I_1+I_2 \right),\notag
 \end{align}
 where 
 $$I_i:=\left|\int_{A_i}\frac{\varphi(x+y)+\varphi(x-y)-2\varphi(x)}{|y|^{n+2\sigma}}dy\right|\quad i=1,2.$$
 Noticing that on $A_1$
 $$|\varphi(x+y)+\varphi(x-y)-2\varphi(x)|\leq \|D^2\varphi\|_{L^\infty(B_\frac{|x|}{2}(x))}|y|^2,$$ we get 
 $$I_1\leq \|D^2\varphi\|_{L^\infty(B_\frac{|x|}{2}(x))}\int_{A_1}\frac{dy}{|y|^{n-2+2\sigma}}\leq C\|D^2\varphi\|_{L^\infty(B_\frac{|x|}{2}(x))}|x|^{2-2\sigma}.$$
 On the other hand 
 \begin{align}
  I_2\leq 2|\varphi(x)|\int_{A_2}\frac{dy}{|y|^{n+2\sigma}}+2\left|\int_{A_2}\frac{\varphi(x-y)}{|y|^{n+2\sigma}}dy\right|&\leq 2\left|\int_{A_2}\frac{\varphi(x-y)}{|y|^{n+2\sigma}}dy\right|+
  C|\varphi(x)||x|^{-2\sigma}  \notag\\&=: 2I_3+C|\varphi(x)||x|^{-2\sigma}\notag.
 \end{align}
Changing the variable $y\mapsto x-y$ we have 
\begin{align}
I_3=\left|\int_{|x-y|>\frac{|x|}{2}}\frac{\varphi(y)}{|x-y|^{n+2\sigma}}dy\right|&\leq \left|\int_{|x-y|>\frac{|x|}{2},|y|>\frac{|x|}{2}}\frac{\varphi(y)}{|x-y|^{n+2\sigma}}dy\right|
            +\left|\int_{|y|<\frac{|x|}{2}}\frac{\varphi(y)}{|x-y|^{n+2\sigma}}dy\right| \notag\\
      &\leq \left|\int_{|y|<\frac{|x|}{2}}\frac{\varphi(y)}{|x-y|^{n+2\sigma}}dy\right|+C\|\varphi\|_{L^\infty(A_2)}|x|^{-2\sigma}\notag\\
      &=:I_4+C\|\varphi\|_{L^\infty(A_2)}|x|^{-2\sigma}.\notag
\end{align}
 Finally, to bound $I_4$ we use the fact that $\varphi\in S_k$. Setting $f(x)=\frac{1}{|x|^{n+2\sigma}}$
 and using $$\sum_{|\alpha|\leq k}\frac{D^\alpha f(x)}{\alpha!}\int_{\rn}y^{\alpha}\varphi(y)dy=0,\quad x\neq 0,$$ we obtain
 \begin{align}&\int_{|y|<\frac{|x|}{2}}\frac{\varphi(y)}{|x-y|^{n+2\sigma}}dy\notag\\ 
&=\int_{|y|<\frac{|x|}{2}}\frac{\varphi(y)}{|x-y|^{n+2\sigma}}dy -
            \sum_{|\alpha|\leq k}\frac{D^\alpha f(x)}{\alpha!}\int_{|y|<\frac{|x|}{2}}y^{\alpha}\varphi(y)dy-
            \sum_{|\alpha|\leq k}\frac{D^\alpha f(x)}{\alpha!}\int_{|y|>\frac{|x|}{2}}y^{\alpha}\varphi(y)dy \notag\\
&= \int_{|y|<\frac{|x|}{2}}\varphi(y)\left(f(x-y)-\sum_{|\alpha|\leq k}y^{\alpha}\frac{D^\alpha f(x)}{\alpha!}\right)dy-
              \sum_{|\alpha|\leq k}\frac{D^\alpha f(x)}{\alpha!}\int_{|y|>\frac{|x|}{2}}y^{\alpha}\varphi(y)dy \notag\\
&= \int_{|y|<\frac{|x|}{2}}\varphi(y)\sum_{|\beta|= k+1}y^{\beta}R_{\beta}(\xi_y)dy- \sum_{|\alpha|\leq k}\frac{D^\alpha f(x)}{\alpha!}\int_{|y|>\frac{|x|}{2}}y^{\alpha}\varphi(y)dy,\notag
\end{align}
where $R_{\beta}(\xi_y)$ satisfies 
$$f(x-y)=\sum_{|\alpha|\leq k}y^{\alpha}\frac{D^\alpha f(x)}{\alpha!}+\sum_{|\beta|= k+1}y^{\beta}R_{\beta}(\xi_y),\quad |y|<\frac{|x|}{2},\,\xi_y\in B_{\frac{|x|}{2}}(x),$$ and
 $$|R_{\beta}(\xi_y)|\leq C\max_{|\alpha|=k+1}\max_{z\in B_{\frac{|x|}{2}}(x)}|D^{\alpha}f(z)|\leq \frac{C}{|x|^{n+2\sigma+k+1}}.$$
Therefore,
\begin{align}
 I_4&\leq \sum_{|\beta|= k+1}\int_{|y|<\frac{|x|}{2}}|\varphi(y)||y|^{|\beta|}|R_{\beta}(\xi_y)|dy+
  \sum_{|\alpha|\leq k}\frac{|D^\alpha f(x)|}{\alpha!}\int_{A_2}|y|^{|\alpha|}|\varphi(y)|dy\notag\\
 &\leq \frac{C}{|x|^{n+2\sigma+k+1}}\int_{\rn}|\varphi(y)||y|^{k+1}dy+\|\sqrt{|\varphi|}\|_{L^\infty(A_2)}\sum_{|\alpha|\leq k}\frac{|D^\alpha f(x)|}{\alpha!}\int_{\rn}|y|^{|\alpha|}\sqrt{|\varphi(y)|}dy,\notag
\end{align}
 and complete the proof.
 \end{proof}

\begin{lem}\label{vweak}
Let $f\in L^1(\rn)$. We set   
\begin{align}
\tilde{v}(x)=\frac{1}{\gamma_n}\int_{\rn}\log\left(\frac{1+|y|}{|x-y|}\right)f(y)dy,\quad x\in\rn . \label{v2}
\end{align}
Then 
\begin{itemize}
\item[(i)] $\tilde{v}\in W^{n-1,1}_{loc}(\rn)$ and 
$$D^{\alpha}\tilde{v}=\frac{1}{\gamma_n}\int_{\rn}D^{\alpha}_x\log\left(\frac{1+|y|}{|x-y|}\right)f(y)dy,\quad 0\leq|\alpha|\leq n-1.$$

\item[(ii)] $D^{\alpha}\tilde{v}\in L_{\frac{1}{2}}(\rn)$ for every  multi-index $\alpha\in\mathbb{N}^n$ with $0\leq|\alpha|\leq n-1$.
\item[(iii)] For every $\varphi\in\s(\rn)$ 
$$\int_{\rn}\tilde{v}(x)(-\Delta)^{\frac{n}{2}}\varphi(x)dx=\int_{\rn}(-\Delta)^{\frac{n-1}{2}}\tilde{v}(x)(-\Delta)^{\frac{1}{2}}\varphi(x)dx=\int_{\rn}\varphi(x)f(x)dx,$$ 
that is $\tilde{v}$ solves \eqref{eqf} in the sense of Definition \ref{def-soln} and \ref{def-new}.
\end{itemize}
\end{lem}
\begin{proof}
Proof of $(i)$ is trivial. 

To prove $(ii)$ first we consider $0<|\alpha|\leq n-1$ and we estimate
 \begin{align}
  &\int_{\rn}\frac{|D^{\alpha}\tilde{v}(x)|}{1+|x|^{n+1}}dx \notag\\&\leq C\int_{\rn}|f(y)|\left(\int_{\rn}\frac{1}{(1+|x|^{n+1})|x-y|^{|\alpha|}}dx\right)dy \notag\\
  &=C\int_{\rn}|f(y)|\left(\int_{B_1(y)}\frac{dx}{(1+|x|^{n+1})|x-y|^{|\alpha|}}+\int_{\rn\setminus B_1(y)}\frac{dx}{(1+|x|^{n+1})|x-y|^{|\alpha|}}\right)dy
                                     \notag\\
  & \leq C \int_{\rn}|f(y)|\left(\int_{B_1(y)}\frac{dx}{|x-y|^{|\alpha|}}+\int_{\rn\setminus B_1(y)}\frac{dx}{(1+|x|^{n+1})}\right)dy   \notag \\
  & <\infty. \notag
 \end{align}
 The case when $\alpha=0$ follows from 
    \begin{align}
   &\int_{\rn}\frac{|\tilde{v}(x)|}{1+|x|^{n+1}}dx \leq \frac{1}{\gamma_n}\int_{\rn}\frac{1}{1+|x|^{n+1}}\left(\int_{\rn}\left|\log\frac{1+|y|}{|x-y|}\right||f(y)|dy\right)dx \notag \\
  &=\frac{1}{\gamma_n}\int_{\rn}|f(y)|\left(\int_{|x-y|>1}\frac{1}{1+|x|^{n+1}}\left|\log\frac{1+|y|}{|x-y|}\right|dx+
  \int_{|x-y|<1}\frac{1}{1+|x|^{n+1}}\left|\log\frac{1+|y|}{|x-y|}\right|dx\right)dy \notag \\   
&\leq \frac{1}{\gamma_n}  \int_{\rn}|f(y)|\left(\int_{|x-y|>1}\frac{\log(2+|x|)}{1+|x|^{n+1}}dx+
  \int_{|x-y|<1}\left(\frac{\log(2+|x|)}{1+|x|^{n+1}}+\left|\log|x-y|\right|dx\right)\right)dy \notag \\
&=  \frac{1}{\gamma_n}  \int_{\rn}|f(y)|\left(\int_{\rn}\frac{\log(2+|x|)}{1+|x|^{n+1}}dx +\|\log(\cdot)\|_{L^1(B_1)} \right)dy \notag \\
&<\infty,\notag
   \end{align}
where in the first inequality we used
   $$\frac{1}{1+|x|}\leq \frac{1+|y|}{|x-y|}\leq 2+|x|,\, 1+|y|\leq 2+|x| \text{ for } |x-y|\geq1.$$ 
 $(iii)$ follows from integration by parts and Lemma \ref{funda}.
\end{proof}

\begin{lem}\label{classi-new}
Let $u$ be a solution of \eqref{eqf} with $f\in L^1(\rn)$ in the sense of Definition \ref{def-new}. Let $\tilde{v}$ be given by \eqref{v2}. Then $p:=u-\tilde{v}$ is a polynomial of degree at most $n-1$.
\end{lem}
\begin{proof}
 Let us consider a function $\psi\in C_c^{\infty}({\rn\setminus\{0\}})$. We set 
$$\varphi:=\mathcal{F}^{-1}\left(\frac{\bar{\psi}}{|\xi|^n}\right)\in\s(\rn),\quad \bar{\psi}(x):=\psi(-x),\,x\in\rn.$$
Now the growth assumption to $u$ in Definition \ref{def-new} implies that $u$ is a tempered distribution and at the same time the function $v$ 
is also a tempered distribution thanks to Lemma \ref{vweak}. Therefore $p\in L_{\frac{n}{2}}(\rn)$ and $\hat{p}\in\s'(\rn)$. Indeed, 
\begin{align}
\langle \hat{p},\psi\rangle=\int_{\rn}p\hat{\psi}dx=\int_{\rn}p(x)(-\Delta)^{\frac{n}{2}}\varphi(x)dx=0, \notag
\end{align}
where the last equality follows from the Definition \ref{def-new} and Lemma \ref{vweak}.

Thus $\hat{p}$ is a tempered distribution with support $\hat{p}\subseteq\{0\}$ which implies that $p$ is a polynomial and combining with $p\in L_{\frac{n}{2}}(\rn)$ we conclude
 that degree of $p$ is at most $n-1$.
 \end{proof} 

\begin{lem}\label{classi-old}
 Let $u$ be a solution of \eqref{eqf} with $f\in L^1(\rn)$ in the sense of Definition \ref{def-soln} and let $\tilde{v}$ be given by \eqref{v2}.
 If $u$ also satisfies 
 \begin{align}\label{deg}
  \int_{B_R}u^+dx=o(R^{2n})\quad \text{or }\int_{B_R}u^-dx=o(R^{2n}) \quad\text{as $R\to\infty$},
 \end{align}
  then $p:=u-\tilde{v}$ is a polynomial of degree at most $n-1$.
\end{lem}
\begin{proof}
 We have $\Delta^{\frac{n-1}{2}}p\in L_{\frac 12}(\rn)$ and it satisfies
\begin{align}\label{pol}
 \int_{\rn}(-\Delta)^{\frac{n-1}{2}}p(-\Delta)^\frac12\varphi dx=0,\quad\text{ for every }\varphi\in\s(\rn),
\end{align}
thanks to Lemma \ref{vweak}. Moreover, by 
 Schauder's estimate (see e.g. \cite[Proposition 22]{TALJ}) for some $\alpha>0$
 $$\|(-\Delta)^{\frac{n-1}{2}}p\|_{C^{0,\alpha}(B_1)}\leq C\|(-\Delta)^{\frac{n-1}{2}}p\|_{L_{\frac12}(\rn)}.$$ Adapting the arguments in \cite[Lemma 15]{TALJ}
 one can get that $(-\Delta)^{\frac{n-1}{2}}p$ is constant in $\rn$ and hence $(-\Delta)^{\frac{n+1}{2}}p=0$ in $\rn$. 
 Noticing that $v\in L_\frac n2(\rn)$ we conclude the proof by Lemma \ref{deg-pol} below.
\end{proof}

\begin{prop}\label{defequiv}
Let $f\in L^1(\rn)$. Then the following are equivalent: 
\begin{itemize}
 \item [(i)] $u$ is a solution of \eqref{eqf} in the sense of Definition \ref{def-new}.
 \item [(ii)] $u$ is a solution of \eqref{eqf} in the sense of Definition \ref{def-soln} and $u$ satisfies \eqref{deg}. 
\end{itemize}
In particular, Definition \ref{def-soln} and Definition \ref{def-new} are equivalent for the solutions of  \eqref{main-eq}-\eqref{V}.
\end{prop}
\begin{proof}
If $p$ is a polynomial of degree at most $n-1$ then $p\in L_{\frac n2}(\rn)$ and 
$$\int_{\rn}p(-\Delta)^\frac n2\varphi dx=\int_{\rn}p(-\Delta)^\frac {n-1}{2}(-\Delta)^\frac12\varphi dx=C_p\int_{\rn}(-\Delta)^\frac12\varphi dx=0, \quad\varphi\in\s(\rn),$$
where $C_p:=(-\Delta)^{\frac{n-1}{2}}p$ is a constant and the second equality follows from integration by parts (which can be justified thanks to Lemma \ref{S_k}). 
Now the equivalence of $(i)$ and $(ii)$ follows immediately from Lemmas \ref{vweak}, \ref{classi-new} and \ref{classi-old}. To conclude the lemma notice that the condition
 \eqref{V} implies $$\int_{B_R}u^+dx=\frac 1n\int_{B_R}nu^+dx\leq \frac 1n\int_{B_R}e^{nu}dx\leq \frac Vn.$$ 
 \end{proof}

\subsection{Proof of Theorem \ref{regularity}}
First we write $(n-1)!e^{nu}=f_1+f_2$ where $f_1\in L^1(\rn)\cap L^\infty(\rn)$ and  $f_2\in L^1(\rn)$.
Let us define the functions 
      $$u_i(x):=\frac{1}{\gamma_n}\int_{\rn}\log\left(\frac{1+|y|}{|x-y|}\right)f_i(y)dy,\quad x\in\rn,\,i=1,2.$$ 
      Then we have that $u_1\in C^{n-1}(\rn)$ and 
  $u_2\in W^{n-1,1}_{loc}(\rn)$. Indeed, for $p\in\left(0,\frac{\gamma_n}{\|f_2\|}\right)$ using Jensen's inequality 
  \begin{align}
   \int_{B_R}e^{np|u_2|}dx&=\int_{B_R}\ex\left(\int_{\rn}\frac{np\|f_2\|}{\gamma_n}\log\left(\frac{1+|y|}{|x-y|}\right)\frac{f_2(y)}{\|f_2\|}dy\right)dx\notag\\
   &\leq \int_{B_R}\int_{\rn}\ex\left(\frac{np\|f_2\|}{\gamma_n}\log\left(\frac{1+|y|}{|x-y|}\right)\right)\frac{|f_2(y)|}{\|f_2\|}dydx\notag\\
   &=\frac{1}{\|f_2\|}\int_{\rn}|f_2(y)|\int_{B_R}\left(\frac{1+|y|}{|x-y|}\right)^\frac{np\|f_2\|}{\gamma_n}dxdy\notag\\
   &\leq C(n,p,\|f_2\|,R),\label{br}
  \end{align}
  where  $\|\cdotp\|$ denotes the  $L^1(\rn)$ norm. Moreover, by Lemma \ref{vweak} (with $\tilde{v}=u_i$ and $f=f_i$) we have 
  $$\int_{\rn}(-\Delta)^{\frac{n-1}{2}}u_i(-\Delta)^{\frac{1}{2}}\varphi dx=\int_{\rn}f_i\varphi dx, \quad \text{for every } \varphi\in\s.$$
          We set $$u_3:=u-u_1-u_2.$$
   We claim that the function $u_3$ is smooth in $\rn$ whenever $u$ is a solution of \eqref{main-eq}-\eqref{V} in the sense of Definition \ref{def-soln} or \ref{def-new}.
   Then  taking \eqref{br} into account we have $e^{nu}\in L^p_{loc}(\rn)$ for every $p<\infty$ and hence $f_2\in L^p_{loc}(\rn)$ . Therefore, for every $x\in B_R$  by 
   H\"older's inequality 
   \begin{align}
    | u_2(x)|&\leq C\int_{|y|<2R}\left|\log\left(\frac{1+|y|}{|x-y|}\right)\right||f_2(y)|dy+C\int_{|y|\geq 2R}\left|\log\left(\frac{1+|y|}{|x-y|}\right)\right||f_2(y)|dy\notag\\
    &\leq C\left(\log(1+2R)\|f_2\|_{L^1(B_{2R})}+\|\log(\cdotp)\|_{L^2(B_{3R})}\|f_2\|_{L^2(B_{2R})}\right) +C\log(3R)\|f_2\|_{L^1(B_{2R}^c)},\notag
   \end{align}
and for every $0<|\alpha|\leq n-1$ again by H\"older's inequality 
   \begin{align}
    |D^\alpha u_2(x)|&\leq C\int_{|y|<2R}\frac{1}{|x-y|^{|\alpha|}}|f_2(y)|dy+C\int_{|y|\geq 2R}\frac{1}{|x-y|^{|\alpha|}}|f_2(y)|dy\notag\\
    &\leq C\||(\cdotp)|^{-|\alpha|}\|_{L^p(B_{3R})}\|f_2\|_{L^{p'}(B_{2R})}+CR^{-|\alpha|}\|f_2\|_{L^1(B_{2R}^c)},\notag
   \end{align}
where $p\in (1,\frac{n}{n-1})$. 
   Thus $u_2\in W^{n-1,\infty}_{loc}(\rn)$ and by Sobolev embeddings we have $u_2\in C^{n-2}(\rn)$, which implies that $u=u_1+u_2+u_3\in C^{n-2}(\rn)$.
      Now to prove $u\in C^\infty(\rn)$ we proceed by induction.
      
      Set $\tilde{u}=u_1+u_2$. Then for $0<|\alpha|\leq n-1$
   $$D^\alpha\tilde{u}(x)=\frac{(n-1)!}{\gamma_n}\int_{\rn}D^\alpha_x\log\left(\frac{1+|y|}{|x-y|}\right)e^{nu(y)}dy
   =:\int_{\rn}K_\alpha(x-y)e^{nu(y)}dy,\quad x\in\rn.$$
   Notice that the function $K_\alpha$ is smooth in $\rn\setminus\{0\}$ and it also satisfies the estimate 
   $$|D^\beta K_\alpha(x)|\leq \frac{C_\alpha}{|x|^{|\alpha|+|\beta|}}, \quad \beta\in\mathbb{N}^n,\, x\in\rn\setminus\{0\}.$$ 
   We rewrite the function $D^\alpha\tilde{u}(x)$ as 
  \begin{align}
   D^\alpha\tilde{u}(x)&=\int_{\rn}\eta(x-y)K_\alpha(x-y)e^{ny(y)}dy+\int_{\rn}(1-\eta(x-y))K_\alpha(x-y)e^{nu(y)}dy\notag\\
   &=\int_{\rn}\eta(x-y)K_\alpha(x-y)e^{ny(y)}dy+\int_{\rn}(1-\eta(y))K_\alpha(y)e^{nu(x-y)}dy,\notag
  \end{align}
where    $\eta\in C^\infty(\rn)$ satisfies
   $$\eta(x)=\left\{\begin{array}{ll}
                     0 &\text{ if }|x|\leq1\\
                     1 &\text{ if }|x|\geq 2.
                    \end{array}\right.
$$
If we assume $u\in C^k(\rn)$ for some integer $k\geq 1$ then 
observing that $\eta K_\alpha\in C^\infty(\rn)$, $D^\beta(\eta K_\alpha)\in L^\infty(\rn)$ and  $1-\eta$ is compactly supported, one has 
$$ D^{\alpha+\beta}\tilde{u}(x)=\int_{\rn}D^\beta_x(\eta(x-y)K_\alpha(x-y))e^{ny(y)}dy+\int_{\rn}(1-\eta(y))K_\alpha(y)D^\beta_xe^{nu(x-y)}dy,\quad |\beta|\leq k. $$
Thus $u\in C^{k+n-1}(\rn)$ thanks to the claim that $u_3\in C^\infty(\rn)$, which proves our induction argument.  

It remains to show that $u_3\in C^\infty(\rn)$ whenever $u$ is a solution of \eqref{main-eq}-\eqref{V} in the sense of Definition \ref{def-soln} or \ref{def-new}.

In the case of Definition \ref{def-new} from Lemma \ref{classi-new} we have that $u_3$ is a polynomial of degree at most $n-1$ 
and hence it is smooth. On the other hand, if we consider Definition \ref{def-soln} then by Lemma \ref{vweak} we get $\Delta^{\frac{n-1}{2}}u_3\in L_{\frac12}(\rn)$ and 
it also satisfies \eqref{pol} with $p=u_3$. Therefore, by \cite[Proposition 2.22]{LS} we have $\Delta^{\frac{n-1}{2}}u_3\in C^\infty(\rn)$ which implies that $u_3\in C^\infty(\rn)$.
    \hfill $\square$\\

\section{Classification of solutions}

\subsection{A fractional version of a lemma of Br\'ezis and Merle}

Theorem \ref{Lp} below is a fractional version of a lemma of Br\'ezis and Merle \cite[Theorem 1]{BM}, compare also \cite[Theorem 5.1]{DMR},
which we shall later need in the proof of Lemma \ref{v-upper}. Although, in our case Theorem \ref{Lp} will be used in a smooth setting, here we shall prove it with  more generality because of 
its independent interest. Before stating the theorem we need the following definition, partially inspired by \cite[Section 3.3]{abatangelo}.

 \begin{defn}\label{def-11} Let $\Omega$ be a smooth bounded domain in $\rn$.
 Assume $f\in L^1(\Omega)$ and $g_j\in L^1(\partial \Omega)$ for $j=0,1,...,\frac{n-3}{2}$. We say that $w\in L_{\frac 12}(\rn)$ 
is a solution of 
 \begin{align}\label{soln-11}
 \left\{\begin{array}{ll}
   (-\Delta)^{\frac{n-1}{2}}(-\Delta)^\frac 12 w=f & in \, \Omega\\
   (-\Delta)^j(-\Delta)^\frac 12 w=g_j & on \, \partial \Omega,\, j=0,1,...,\frac{n-3}{2}\\
 \end{array}
\right.
  \end{align}
  if $w$ satisfies 
  \begin{align}\label{d1}
   \int_{d(x,\partial\Omega)<2,x\in\Omega^c}\frac{|w(x)|}{\sqrt{\delta(x)}}dx<\infty,
  \end{align}
 and  there exists a function $W\in L^1(\Omega)$ such that $(-\Delta)^\frac 12w=W$ in $\Omega$, i.e.
 \begin{align}\label{d2}
  \int_{\rn}w(-\Delta)^\frac 12\varphi dx=\int_{\Omega}W\varphi dx\quad \text{for every }\varphi\in T_1,
 \end{align}
  and the function $W$ satisfies
  \begin{align}\label{L1soln}
  \left\{\begin{array}{ll}
   (-\Delta)^{\frac{n-1}{2}} W=f & in \, \Omega\\
   (-\Delta)^j W=g_j & on \, \partial \Omega,\, j=0,1,...,\frac{n-3}{2},\\
 \end{array}
\right.
  \end{align}
  i.e. $$\int_{\Omega}W(-\Delta)^\frac{n-1}{2}\varphi dx=\int_{\Omega}f\varphi dx-
  \sum_{j=0}^{\frac{n-3}{2}}\int_{\partial \Omega}g_j\frac{\partial}{\partial\nu}(-\Delta)^{\frac{n-3}{2}-j}\varphi d\sigma
  \quad \text{for every }\varphi\in T_2,$$ where the spaces of test functions 
  $T_1$ and $T_2$ are defined by 
  \begin{align}
  T_1:=\left\{\varphi\in C^{\infty}(\Omega)\cap C^{\frac 12}(\rn):
  \left\{\begin{array}{ll}
  (-\Delta)^\frac 12 \varphi=\psi &\, \text{ in } \Omega\\
  \varphi=0 &\,\text{ on } \Omega^c
 \end{array}
\right.\text{ for some }\psi\in C_c^{\infty}(\Omega),\right\}, \notag 
   \end{align}
   and 
   $$T_2:=\left\{\varphi\in C^{n-1}(\overline{\Omega}):\Delta^j\varphi=0\text{ on }\partial\Omega,\,j=0,1,\dots,\frac{n-3}{2}\right\}.$$
   \end{defn}
   Notice that the left hand side of \eqref{d2} is well-defined thanks to the assumption \eqref{d1} and Lemma \ref{est-1} below.

 \begin{lem}[Maximum Principle]
  Let $w$ be a solution of \eqref{soln-11} with $f,g_j\geq0$ in the sense of Definition \ref{def-11}. If $w\geq 0$ on $\Omega^c$ then $w\geq 0$ in $\Omega$.
 \end{lem}
\begin{proof}
 First notice that the conditions $f\geq 0$, $g_j\geq 0$ implies that $W\geq 0$ in $\Omega$, where $W\in L^1(\Omega)$ is a solution of \eqref{L1soln}. Now consider
 a test function $\psi\in C_c^{\infty}(\Omega)$ such that $\psi\geq0$ in $\Omega$. Let $\varphi\in T_1$ be the solution of 
 $(-\Delta)^\frac 12\varphi=\psi$ in $\Omega$. Then by classical maximum principle one has $\varphi\geq 0$ in $\Omega$. Since the constant $C_{n,\frac 12}>0$ 
 in Proposition \ref{value} we get 
 $$(-\Delta)^\frac 12\varphi(x)< 0\quad \text{ for }x\in\rn\setminus\overline{\Omega},$$ and from \eqref{d2} 
  $$\int_{\Omega}w\psi dx=\int_{\Omega}w(-\Delta)^\frac 12\varphi dx=\int_{\Omega}W\varphi dx-\int_{\Omega^c}w(-\Delta)^\frac 12\varphi dx\geq0,$$ which completes the proof.
 \end{proof}

\begin{thm}\label{Lp}
 Let $f\in L^1(B_R)$. Let $u\in L^1(B_R)$ be a solution of \eqref{soln-11} (in the sense of Definition \ref{def-11})  with $g_j=0$ for $j=0,1,...,\frac{n-3}{2}$ and $u=0$ on $B_R^c$.
 Then for any  $p\in\left(0,\frac{\gamma_n}{\|f\|_{L^1(B_R)}}\right)$  $$\int_{B_R}e^{np|u|}dx\leq C(p,R).$$
\end{thm}
\begin{proof}
 We set $$\overline{W}(x)=\int_{B_R}\Psi(x-y)|f(y)|dy\quad x\in \rn,$$ where 
 $$\Psi(x):=\frac{\Gamma(\frac 12)}{n2^{n-2}|B_1|\Gamma(\frac n2)\left(\frac{n-3}{2}\right)!}\frac{1}{|x|},$$ is a fundamental solution of 
 $(-\Delta)^{\frac{n-1}{2}}$ in $\rn$ (see \cite[Section 2.6]{Gazzola}). Then $\overline{W}\in L^1(B_R)$ satisfies 
  $$\left\{\begin{array}{ll}
   (-\Delta)^{\frac{n-1}{2}}\overline{W}=|f|& in \, B_R\\
   (-\Delta)^j\overline{W}\geq0 & on \, \partial B_R,\, j=0,1,...,\frac{n-3}{2},\\ 
       \end{array}
\right.
  $$
 and by maximum principle $\overline{W}\geq |W|$ in $B_R$, where $W\in L^1(B_R)$ is a solution of \eqref{L1soln}. Let us define
 $$\overline{u}(x):=\Phi*(\overline{W}\chi_{B_R})(x)=\frac{(\frac{n-3}{2})!}{2\pi^{\frac{n+1}{2}}}\int_{\rn}\frac{1}{|x-y|^{n-1}}\overline{W}(y)\chi_{B_R}(y)dy,\quad x\in\rn,$$ 
 where $\Phi$ is given in Lemma \ref{funda} below. Noticing 
 $$\frac{1}{\gamma_n}=|S^{n-1}|\frac{\Gamma(\frac 12)}{n2^{n-2}|B_1|\Gamma(\frac n2)\left(\frac{n-3}{2}\right)!}\frac{(\frac{n-3}{2})!}{2\pi^{\frac{n+1}{2}}},$$
 in view of Lemma \ref{p3} below one has 
 $$|\overline{u}(x)|\leq C+\frac{1}{\gamma_n}\int_{|y|<R}|f(y)||\log|x-y||dy, \quad x\in\rn,$$ 
  which yields
$$\overline{u}\in L^q_{loc}(\rn)\cap L^{\infty}(\rn\setminus B_{R+\delta}),\quad q\in[1,\infty),\,\delta>0.$$
Moreover, for every $\varphi\in\s(\rn)$
\begin{align}\label{test}
 \int_{B_R}\overline{W}\varphi dx=\int_{\rn}\overline{u}(-\Delta)^\frac 12\varphi dx=\int_{B_R}\overline{u}(-\Delta)^\frac 12\varphi dx+
 \int_{B_R^c}\overline{u}(-\Delta)^\frac 12\varphi dx,
\end{align}
thanks to Lemma \ref{funda} below.

We claim that \eqref{test} holds for $\varphi\in T_1$. Then for any $\varphi\in T_1$ with $\varphi\geq 0$
$$\int_{B_R}(\overline{u}\pm u)(-\Delta)^\frac 12\varphi dx=\int_{B_R}\underbrace{(\overline{W}\pm W)}_{\geq 0}\varphi dx-
 \int_{B_R^c}\overline{u}\underbrace{(-\Delta)^\frac 12\varphi}_{\leq 0} dx\geq 0,$$ and by maximum principle one has $\overline{u}\geq |u|$ in $B_R$
 and the lemma follows at once.

To prove the claim we consider a mollifying sequence $\varphi_k:=\varphi*\rho_k$, where $\rho_k(x)=k^{n}\rho(kx)$. 
Then (see \cite[Section A]{abatangelo})
\begin{align}
(-\Delta)^\frac 12\varphi_k(x)=\varphi*(-\Delta)^\frac 12\rho_k(x)\quad x\in\rn,\notag
\end{align}
and 
\begin{align}\label{a2}
(-\Delta)^\frac 12\varphi_k(x)=\rho_k*(-\Delta)^\frac 12\varphi(x),\quad dist(x,\partial B_R)>\frac 1k.
\end{align}
Then the uniform convergence of 
$\varphi_k$ to $\varphi$ imply $$\int_{B_R}\overline{W}\varphi_k dx\xrightarrow{k\to\infty} \int_{B_R}\overline{W}\varphi dx.$$ 
Using the uniform convergence of $(-\Delta)^\frac 12\varphi_k$ to $(-\Delta)^\frac 12\varphi$ on the compact sets in $B_R$ and the fact that 
 $\supp (-\Delta)^\frac 12\varphi|_{B_R}\subseteq B_R$ we get
$$\int_{B_R}\overline{u}(-\Delta)^\frac 12\varphi_k dx\xrightarrow{k\to\infty}\int_{B_R}\overline{u}(-\Delta)^\frac 12\varphi dx.$$
It remains to verify that
\begin{align}
\int_{B_R^c}\overline{u}(-\Delta)^\frac 12\varphi_k dx\xrightarrow{k\to\infty}\int_{B_R^c}\overline{u}(-\Delta)^\frac 12\varphi dx,\notag
\end{align}
 which follows immediately from 
\begin{align}\label{a3}
 (-\Delta)^\frac 12\varphi_k\xrightarrow{k\to\infty} (-\Delta)^\frac 12\varphi\text{ in }L^q(B_{R+1}\setminus B_R),\text{ for some }q>1,
\end{align}
 and 
\begin{align}\label{a4}
 (-\Delta)^\frac 12\varphi_k\xrightarrow{k\to\infty} (-\Delta)^\frac 12\varphi\text{ in }L^1(B_{R+1}^c).
\end{align}
With the help of Lemma \ref{est-1} below and \eqref{a2} one can get \eqref{a4}.
 To conclude \eqref{a3} first notice that $(-\Delta)^\frac 12\varphi_k$ converges to $(-\Delta)^\frac 12\varphi$ point-wise and that $(-\Delta)^\frac 12\varphi\in L^q(B_{R+1}\setminus B_R)$ for any $q\in[1,2)$ thanks to
 Lemma \ref{est-1} below. By  \cite[Theorem 1.9 (Missing term in Fatou's lemma)]{Lieb} it is sufficient to show that for some $q>1$
$$\int_{R<|x|<R+1}|(-\Delta)^\frac 12\varphi_k(x)|^qdx\leq \int_{R<|x|<R+1}|(-\Delta)^\frac 12\varphi(x)|^qdx+o(1),$$ where $o(1)\to 0$ as $k\to\infty$.
Now using the estimate (see for instance \cite[Section A]{abatangelo})
$$|(-\Delta)^\frac 12\rho_k(x)|\leq Ck^{n+1}\quad x\in\rn,$$ 
and  fixing $t$ and $q$ such that
  $$\frac{2n}{2n+1}<t<1,\quad 1<q<\min\left\{\frac{1+nt}{t+nt},\frac{2nt+t+2}{2n+2}\right\},$$  we bound
\begin{align}
 &\int_{R<|x|<R+1}|(-\Delta)^\frac 12\varphi_k(x)|^qdx=\int_{R<|x|<R+\frac 1k}|(-\Delta)^\frac 12\varphi_k(x)|^qdx+\int_{R+\frac 1k<|x|<R+1}|(-\Delta)^\frac 12\varphi_k(x)|^qdx\notag\\
 &=\int_{R<|x|<R+\frac 1k}|\varphi*(-\Delta)^\frac 12\rho_k(x)|^qdx+\int_{R+\frac 1k<|x|<R+1}|\rho_k*(-\Delta)^\frac 12\varphi(x)|^qdx\notag\\
&\leq \|\varphi\|_{L^1}^{q-1}\int_{R<|x|<R+\frac 1k}\int_{\rn}|(-\Delta)^\frac 12\rho_k(y)|^q|\varphi(x-y)|dydx\notag\\
&                          \quad  +\int_{R+\frac 1k<|x|<R+1}\int_{\rn}|(-\Delta)^\frac 12\varphi(y)|^q\rho_k(x-y)dydx  \notag\\
&= \int_{R<|y|<R+1+\frac1k}|(-\Delta)^\frac 12\varphi(y)|^q+\|\varphi\|_{L^1}^{q-1}\int_{R<|x|<R+\frac 1k}\int_{|y|>\frac{1}{k^t}}|(-\Delta)^\frac 12\rho_k(y)|^q|\varphi(x-y)|dydx\notag\\
&   \quad +\|\varphi\|_{L^1}^{q-1}\int_{R<|x|<R+\frac 1k}\int_{|x-y|<R,|y|<\frac{1}{k^t}}|(-\Delta)^\frac 12\rho_k(y)|^q|\varphi(x-y)|dydx
     \notag\\
&\leq \int_{R<|y|<R+1+\frac 1k}|(-\Delta)^\frac 12\varphi(y)|^q+C\|\varphi\|_{L^1}^{q-1}k^{t(q+nq-n)-1}+C\|\varphi\|_{L^1}^{q-1}k^{q(n+1)-nt-\frac t2-1}\notag\\
&=\int_{R<|y|<R+1}|(-\Delta)^\frac 12\varphi(y)|^q+o(1), \notag
 \end{align}
where in the last inequality we have used (for the second term)
  \begin{align}
  \int_{|x|>\frac{1}{k^t}}|(-\Delta)^\frac 12\rho_k(x)|^qdx&=\int_{|x|>\frac{1}{k^t}}\left|C_{1/2}P.V.\int_{\rn}\frac{\rho_k(x)-\rho_k(y)}{|x-y|^{n+1}}dy\right|^qdx\notag\\
  &\leq C\int_{|x|>\frac{1}{k^t}}\int_{|y|<1}\frac{\rho(y)^q}{|x-\frac yk|^{nq+q}}dydx\notag\\
  &\leq C\int_{|y|<1}\int_{|x|>\frac{1}{k^t}}\frac{1}{|x|^{nq+q}}dxdy\notag\\
 &\leq Ck^{t(q+nq-n)}.\notag
 \end{align}
\end{proof}

\begin{lem}\label{p3} Let $\Omega$ be a domain in $\rn$. Let $p$ and $q$ be two positive real numbers.
Then 
 $$\int_{\Omega}\frac{dy}{|x-y|^{n+p}}\leq \frac{|S^{n-1}|}{p}\frac{1}{\delta(x)^p},\,\quad\text{if } \delta(x):=\dist(x,\Omega)>0,$$ and
 $$\int_{\Omega}\frac{dz}{|x-z|^p|y-z|^q}\leq \frac{C_{n,p,q}}{|x-y|^{p+q-n}}, \,\quad\text{if } p+q>n,p<n,q<n, \,x\neq y, $$
  where the constant $C_{n,p,q}$ is given by (an explicit formula can be found in \cite[Section 5.10]{Lieb})
  $$C_{n,p,q}=\int_{\rn}\frac{dz}{|z|^p|e_1-z|^q}.$$ 
  In addition if we also assume that the domain $\Omega$ is bounded then 
  $$\int_{\Omega}\frac{dy}{|x-y|^{n}}\leq |\Omega|+|S^{n-1}||\log\delta(x)| \,\text{ if }\delta(x)>0,$$ and
    $$\int_{\Omega}\frac{dz}{|x-z|^p|y-z|^q}\leq C+|S^{n-1}|\left|\log(|x-y|)\right|, \text{ if }p+q=n,p<n,q<n,\,x\neq y.$$ 
  
\end{lem}

\begin{proof} Let us denote the set $\{y-x:y\in\Omega\}$ by $\Omega-x$.
Using a change of variable $z\mapsto z-x$ and setting $w=y-x$  we have        
\begin{align}
&\int_{\Omega}\frac{dz}{|x-z|^p|y-z|^q}=\int_{\Omega-x}\frac{dz}{|z|^{p}|w-z|^q}=:I.\notag
\end{align}
If $p+q>n$ then changing the variable  $z\mapsto |w|z$ one has
$$I=\frac{1}{|w|^{p+q-n}}\int_{\frac{1}{|w|}(\Omega-x)}\frac{dz}{|z|^{p}|\frac{w}{|w|}-z|^q}\leq \frac{1}{|w|^{p+q-n}}\int_{\rn}\frac{dz}{|z|^{p}|\frac{w}{|w|}-z|^q}=\frac{C_{n,p,q}}{|w|^{p+q-n}}.$$
In the case when $p+q=n$, we split the domain $\Omega-x$ into two disjoint domains:
$$\Omega_1:=\left(\Omega-x\right)\cap B_1,\quad \Omega_2=\left(\Omega-x\right)\cap B_1^c.$$ 
Then $$I=\sum_{i=1}^2I_i,\quad I_i:=\int_{\Omega_i}\frac{dz}{|z|^{p}|w-z|^q}.$$
Since $\Omega_2$ is bounded and $q<n$, we have 
$$I_2\leq \int_{\Omega_2}\frac{dz}{|w-z|^q}\leq C.$$
Now using 
$$\frac{1}{|\frac{w}{|w|}-z|}\leq \frac{1}{|z|}\left(1+\frac{2}{|z|}\right)\quad \text{for }|z|\geq 2,$$ and 
$$\quad (1+x)^q\leq1+C_qx\quad \text{ for }x\in (0,1),$$ 
we bound 
\begin{align}
I_1&\leq\int_{B_1}\frac{dz}{|z|^{p}|w-z|^q}= \int_{|z|\leq \frac{1}{|w|}}\frac{dz}{|z|^p|\frac{w}{|w|}-z|^q}\notag\\
&\leq\underbrace{\int_{|z|\leq 2}\frac{dz}{|z|^p|\frac{w}{|w|}-z|^q}}_{\leq C}+
 \int_{2<|z|\leq \frac{1}{|w|}}\frac{1}{|z|^n}\left(1+\frac{2}{|z|}\right)^qdz\notag\\
 &\leq\int_{2<|z|\leq \frac{1}{|w|}}\frac{1}{|z|^n}\left(1+\frac{C}{|z|}\right)dz\notag\\
 &\leq C+|S^{n-1}||\log|w||.\notag
 \end{align}
Finally,  we conclude the lemma by showing that for $x\in\rn\setminus\overline{\Omega}$
 $$\int_{\Omega}\frac{dy}{|x-y|^{n+p}}\leq\int_{|z|>\delta(x)}\frac{dy}{|z|^{n+p}}= \frac{|S^{n-1}|}{p}\frac{1}{\delta(x)^p},\quad p>0,$$ 
 and 
 $$\int_{\Omega}\frac{dy}{|x-y|^n}\leq |\Omega|+\int_{\Omega\cap B_1(x)}\frac{dy}{|x-y|^n}\leq|\Omega|+\int_{\delta(x)<|z|<1}\frac{dy}{|z|^n}
 =|\Omega|+|S^{n-1}||\log\delta(x)|.$$
 \end{proof}
 
\begin{lem}\label{est-1}
 Let $\Omega$ be a bounded domain in $\rn$. Let $\varphi\in C^{k,\sigma}(\rn)$ for some nonnegative integer $k$ and   $0\leq\sigma\leq 1$ be such that $\varphi=0$ on $\rn\setminus\Omega$. Then
 for $0<s<1$ and for $x\in \rn\setminus\overline{\Omega}$
 $$|(-\Delta)^s\varphi(x)|\leq C\left\{\begin{array}{ll}
                                      \min \{\max\{1,\delta(x)^{-2s+k+\sigma}\},\delta(x)^{-n-2s} \}&\text{ if }k+\sigma\neq 2s\\
                                     \\
                                   \min\{|\log\delta(x)|,\delta(x)^{-n-2s}\}  &\text{ if }k+\sigma=2s,
                                       \end{array}
                                          \right.$$
\end{lem}
where $\delta(x):=\dist(x,\Omega)$.

\begin{proof}
We claim that  $$|\varphi(y)|\leq C|x-y|^{k+\sigma},\quad x\in\rn\setminus \overline{\Omega},\,y\in\Omega,$$ which can be verified using the Taylor's expansion
$$\varphi(y)=\sum_{|\alpha|\leq k-1}\frac{1}{\alpha !}\underbrace{D^{\alpha}\varphi(x)}_{=0}(y-x)^{\alpha}+
\sum_{|\beta|=k}\frac{|\beta|}{\beta !}(y-x)^\beta\int_0^1(1-t)^{|\beta|-1}D^\beta\varphi(x+t(y-x))dt,$$
and  $$|D^\beta\varphi(x+t(y-x))|=|D^\beta\varphi(x+t(y-x))-D^\beta\varphi(x)|\leq C|t(x-y)|^\sigma\leq C|x-y|^\sigma.$$
Therefore,  by Proposition \ref{value}
$$|(-\Delta)^s\varphi(x)|=\left|C_{n,s}\int_{\Omega}\frac{\varphi(y)}{|x-y|^{n+2s}}dy\right|
\leq C\int_{\Omega}\frac{dy}{|x-y|^{n+2s-k-\sigma}},\text{  $x\in\rn\setminus\overline{\Omega}$},$$
and 
$$|(-\Delta)^s\varphi(x)|\leq C\int_{\Omega}\frac{|\varphi(y)|}{|x-y|^{n+2s}}dy\leq C\int_{\Omega}\frac{|\varphi(y)|}{\delta(x)^{n+2s}}dy
\leq\frac{C}{\delta(x)^{n+2s}},\quad x\in\rn\setminus\overline{\Omega}.$$
 Now the proof follows at once from Lemma \ref{p3}.
\end{proof}

\subsection{Proof of Theorem \ref{thm-1}}

First we study the asymptotic behavior of $v$ defined in \eqref{def-v}.

\begin{lem}\label{v-lower}
Let $u$ be a smooth solution of \eqref{main-eq}-\eqref{V} and let $v$ be given by \eqref{def-v}. Then there exists a constant $C>0$ such that 
 $$v(x)\geq -\alpha\log|x|-C,\quad |x|\geq 4.$$
\end{lem}
\begin{proof}
The proof follows as in the proof of \cite[Lemma 2.1]{Lin}.
\end{proof}

A consequence of the above lemma is the following Proposition, compare Lemmas \ref{classi-new}, \ref{classi-old}.
\begin{prop}\label{prop}
Let $u$ be a smooth solution of \eqref{main-eq}-\eqref{V} in the sense of Definition \ref{def-soln} or \ref{def-new} and let $v$ be defined by \eqref{def-v}. Then the function 
$$P(x):=u(x)-v(x), \quad x\in\rn,$$ is a polynomial of degree at most $n-1$
and $P$ is bounded above.
\end{prop}
 \begin{proof}
 Since \eqref{V} implies \eqref{deg}, by Lemmas \ref{classi-new} and \ref{classi-old} 
 we have that $P$ is a polynomial of degree at most $n-1$. On the other hand, using Lemma \ref{v-lower} one can get that $P$ is bounded above 
 (the proof is very similar to \cite[Lemma 11]{LM}.
 \end{proof}
 
 \begin{lem}\label{derivative-v}
Let $n\geq 3$ be an odd integer and let $u$ be a smooth solution of \eqref{main-eq}-\eqref{V} and $v$ be given by \eqref{def-v}.
 Then 
 \begin{itemize}
  \item [(i)] $v\in C^{\infty}(\rn)$ and $D^{\alpha}v\in L_{\frac{1}{2}}(\rn)$ for every multi-index $\alpha\in\mathbb{N}^n$ with $0\leq|\alpha|\leq n-1$.
    \item [(ii)]There exists a constants $C>0$ such that 
    $$\int_{\partial B_4(x)}|(-\Delta)^j(-\Delta)^{\frac{1}{2}}v(y)|d\sigma(y)\leq C \text{ for every }x\in\rn,\,j=0,1,2,...,\frac{n-3}{2}.$$
   \item [(iii)]         
    $v$ is a poitwise solution of $$(-\Delta)^\frac12(-\Delta)^\frac{n-1}{2}v=(n-1)!e^{nu}\quad\text{in }\rn.$$
              
  \item [(iv)]$v$ solves \eqref{soln-11} with $f=(n-1)!e^{nu}$ and $g_j=(-\Delta)^j(-\Delta)^{\frac{1}{2}}v$ for every $ j=0,1,2,\dots,\frac{n-3}{2}$. 
  
 \end{itemize}
\end{lem}
\begin{proof}
We divide the proof into several steps.

\noindent\emph{Step 1}.
From Proposition \ref{prop} we have the smoothness of $v$ and by Lemma \ref{vweak} we get $D^{\alpha}v\in L_{\frac{1}{2}}(\rn)$ 
for every multi-index $\alpha\in\mathbb{N}^n$ with $0\leq|\alpha|\leq n-1$.

\noindent\emph{Step 2}. In this step we use $(i)$ to prove $(ii)$. In fact by Lemmas \ref{comm-3}, \ref{log}, below we have
\begin{align}
\int_{\partial B_4(x)}|(-\Delta)^j(-\Delta)^{\frac{1}{2}}v(y)|d\sigma(y)
&= \int_{\partial B_4(x)}|(-\Delta)^{\frac{1}{2}}(-\Delta)^jv(z)|d\sigma(z)\notag\\
 &\leq C\int_{\partial B_4(x)}\int_{\rn}\frac{e^{nu(y)}}{|y-z|^{2j+1}}dyd\sigma(z)\notag\\
 &=C\int_{\rn}e^{nu(y)}\int_{\partial B_4(x)}\frac{1}{|y-z|^{2j+1}}d\sigma(z)dy\notag\\
 &\leq C.\notag
\end{align}
 
 \noindent\emph{Step 3}. We claim that for $g\in C^\infty(\rn)\cap L_\frac12(\rn)$
 $$\int_{\rn}(-\Delta)^{\frac{1}{2}}g\varphi dx=\int_{\rn}g(-\Delta)^{\frac{1}{2}}\varphi dx \,\text{ for every }\varphi\in C_c^\infty(\rn).$$
   To prove the claim we consider a approximating sequence 
  $$g_k(x):=g(x)\psi(\frac xk),\quad \psi\in C^\infty(\rn),\quad \psi(x)=\left\{\begin{array}{ll}
                                              1 & \text{ if }|x|<1\\
                                              0 &\text{ if }|x|>2.
                                             \end{array}\right.
$$
Then $g_k\in\s(\rn)$ and hence $$\int_{\rn}(-\Delta)^{\frac{1}{2}}g_k\varphi dx=\int_{\rn}g_k(-\Delta)^{\frac{1}{2}}\varphi dx.$$
Now the claim follows from the locally uniform convergence of $(-\Delta)^{\frac{1}{2}}g_k$ to $(-\Delta)^{\frac{1}{2}}g$ and the $L_\frac 12(\rn)$ convergence of
$g_k$ to $g$.

 \noindent\emph{Step 4}. Using  \emph{Step 3} with $g=(-\Delta)^\frac{n-1}{2}v$ we have
 $$\int_{\rn}(-\Delta)^{\frac{1}{2}}(-\Delta)^\frac{n-1}{2}v\varphi dx=\int_{\rn}(-\Delta)^\frac{n-1}{2}v(-\Delta)^{\frac{1}{2}}\varphi dx=(n-1)!\int_{\rn}e^{nu}\varphi dx,$$
 for every $\varphi\in C_c^\infty(\rn)$, which implies $(iii)$. 
 
 To complete $(iv)$ it suffices to show that $W:=(-\Delta)^{\frac{1}{2}}v\in C^\infty(\rn)$ and it satisfies \eqref{d1}-\eqref{L1soln} with $w=v$.
 
 The smoothness of $v$ implies $W\in C^\infty(\rn)$ and \eqref{d1}. Moreover, using integration by
 parts (see \cite[Proposition 1.2.1]{abatangelo}) one can get \eqref{d2}. 
 
 One must notice that the function $u$ in \cite[Proposition 1.2.1]{abatangelo}
 is in $C^{1+\varepsilon}(\Omega)\cap L^\infty(\rn)$ but still we can use it since our function $v\in C^\infty(\rn)\cap L_{\frac 12}(\rn)$.
 
 Finally, we prove \eqref{L1soln} by showing that $W$ is a classical solution of \eqref{L1soln}. Since $W$ is smooth in $\rn$ clearly it satisfies the boundary conditions. 
 Using \emph{step 3} (with $g=v$) and Lemma \ref{vweak} (with $f=(n-1)!e^{nu}$) we have for every $\varphi\in C_c^\infty(\Omega)$ 
 \begin{align}
 \int_{\Omega}(-\Delta)^\frac{n-1}{2}W\varphi dx
& = \int_{\Omega}W(-\Delta)^{\frac {n-1}{2}}\varphi dx =\int_{\rn}(-\Delta)^{\frac{1}{2}}v(-\Delta)^{\frac {n-1}{2}}\varphi dx\notag\\
 &=\int_{\rn}v(-\Delta)^{\frac{1}{2}}(-\Delta)^{\frac {n-1}{2}}\varphi dx\notag
 =(n-1)!\int_{\rn}e^{nu}\varphi dx,\notag
 \end{align}
 that is 
 $$(-\Delta)^\frac{n-1}{2}W=(n-1)!e^{nu}\quad \text{in }\Omega.$$
  \end{proof}

 The following lemma is the crucial part in the proof of Theorem \ref{thm-1}.

\begin{lem}\label{v-upper}
 Let $u$ be a smooth solution of \eqref{main-eq}-\eqref{V} and $v$ be given by \eqref{def-v}. Then for any $\varepsilon>0$ there exists $R>0$ such that for $|x|>R$ 
 $$v(x)\leq (-\alpha+\varepsilon)\log|x|.$$ 
\end{lem}
\begin{proof} 
 \emph{Step 1}. For any $\varepsilon>0$ there exists a $R>0$ such that for $|x|\geq R$
 \begin{align}\label{v+}
  v(x)\leq (-\alpha+\frac{\varepsilon}{2})\log|x|-\frac{(n-1)!}{2}\int_{B_1(x)}\log|x-y|e^{nu(y)}dy.
 \end{align}
 The proof of \eqref{v+} is very similar to the proof of \cite[Lemma 2.4 ]{Lin}. 
As a consequence of \eqref{v+} using Jensen's inequality we have the following estimate
\begin{align}\label{v-Lp}
\|v^+\|_{L^p(\rn)}\leq |\alpha-\frac{\varepsilon}{2}|\|\log\|_{L^p(B_1)}+\frac{(n-1)!}{2}\|e^{nu}\|_{L^1(\rn)}\|\log\|_{L^p(B_1)},\quad 1\leq p<\infty.
\end{align}
 \noindent \emph{Step 2}. We claim that there exists $p>1$ and $C>0$ independent of $x_0$ such that $\|e^{nu}\|_{L^p(B_1(x_0))}\leq C$. Then using 
 H\"{o}lder inequality one can bound the second term on the right hand side of \eqref{v+} uniformly in $x$ and that completes the proof of the lemma.
 
 To prove the claim, first notice that it is sufficient to consider $x_0\in\rn\setminus B_R$ for any fixed $R>0$. We choose $R>0$ large enough such that 
 $$(n-1)!\|e^{nu}\|_{L^1(B_{R-1}^c)}<\frac{\gamma_n}{2}.$$ 
 Let $w\in C^0(\rn)$ be the solution of
$$\left\{\begin{array}{ll}
           (-\Delta)^{\frac{n-1}{2}}(-\Delta)^{\frac{1}{2}}w=(n-1)!e^{nu} & in \,  B_4(x_0)\subset \rn \\
   (-\Delta)^j(-\Delta)^{\frac{1}{2}}w=0 & on \, \partial  B_4(x_0),  for \, j=0,1,...,\frac{n-3}{2} \\
   w=0  & on \,\rn\setminus  B_4(x_0), \\
   \end{array}
\right.  $$ in the sense of Definition \ref{def-11}.
 Since $u$ is smooth by Schauder's estimates and bootstrap argument we have
 $W=(-\Delta)^{\frac{1}{2}}w\in C^\infty(\overline{B_4(x_0)})$  which solves \eqref{soln-11} with $f=(n-1)!e^{nu}$ and $g_j=(-\Delta)^j(-\Delta)^{\frac{1}{2}}v$
 for every $ j=0,1,2,\dots,\frac{n-3}{2}$.  Then using Green's representation formula (see \cite[Theorem 3]{CB}) one can get $w\in C^0(\rn)$ 
 (in fact $w\in C^\frac12(\rn)$, see \cite{Ros}), which is the poitwise continuous unique  solution of 
 $$(-\Delta)^\frac{1}{2}w=W\quad \text{in }B_4(x_0),\quad w=0\quad \text{on }B_4(x_0)^c.$$ Moreover, $w$ satisfies \eqref{d2} thanks to \cite[Proposition 3.3.3]{abatangelo}.
 
We set $h=v-w$. Then we have that $h\in C^0(\rn)$, $(-\Delta)^{\frac{1}{2}}h\in C^{\infty}(\overline{B_4(x_0)})$ and  
\begin{align}\label{def-h}
\left\{\begin{array}{ll}
   (-\Delta)^{\frac{n-1}{2}}(-\Delta)^{\frac{1}{2}}h=0 &in \, B_4(x_0) \\
  (-\Delta)^j(-\Delta)^{\frac{1}{2}}h=(-\Delta)^j(-\Delta)^{\frac{1}{2}}v & on \, \partial B_4(x_0),\, j=0,1,...,\frac{n-3}{2} \\
   h=v &on \,\rn\setminus B_4(x_0),
  \end{array}
\right.
  \end{align} 
  thanks to Lemma \ref{derivative-v}. 
   Indeed, by Lemma \ref{uniform} below there exists a constant $C>0$ independent of the choice of $x_0\in\rn$ such that 
$$h(x)\leq C\quad \text{for every } x\in B_1(x_0).$$ 
Hence by Proposition \ref{prop}$$u=v+P\leq C+h+w\leq C+w,$$ and by Theorem \ref{Lp} we have the proof.
\end{proof}

A simple consequence of Lemma \ref{v-upper} is that 
\begin{align}\label{use}
 \lim_{|x|\to\infty}u(x)=-\infty,
\end{align}
thanks to Proposition \ref{prop}. Using \eqref{use}  one can show that 
 $$\lim_{|x|\to\infty}D^{\beta}v(x)=0 \text{ for every $\beta\in\mathbb{N}^n$ with $0<|\beta|<n-1$ }.$$
Now the proof of Theorem \ref{thm-1} follows at once from Lemmas \ref{v-lower}, \ref{v-upper} and Proposition \ref{prop}.

\begin{lem}\label{uniform}
 Let $h\in C^0(\rn)$ be given by \eqref{def-h}. Then there exists a constant $C>0$ (independent of $x_0$) such that 
 $$h(x)\leq C,\text{ for every }x\in B_1(x_0).$$
 \end{lem}
\begin{proof}
Let us write $h=h_1+h_2$ where $h_1,h_2 \in C^0(\rn)$ be such that
$$
\left\{\begin{array}{ll}
         (-\Delta)^{\frac{1}{2}}h_1 =(-\Delta)^{\frac{1}{2}}h & in\, B_4(x_0)  \\
 h_1 =0 & on \,B_4(x_0)^c, 
        \end{array}
\right. 
$$
and
$$\left\{\begin{array}{ll}
          (-\Delta)^{\frac{1}{2}}h_2 =0 & in\, B_4(x_0)  \\
 h_2=h=v & on \,B_4(x_0)^c. 
         \end{array}
\right.
$$
Let $h_3 \in C^0(\rn)$ be such that
$$\left\{\begin{array}{ll}
 (-\Delta)^{\frac{1}{2}}h_3 =0 & in\, B_4(x_0) \notag \\
 h_3=v^+ & on \,B_4(x_0)^c. \notag
  \end{array}
\right.
$$
Then by maximum principle $$h_2\leq h_3\text{ on }\rn.$$ Without loss of generality we can assume that $x_0=0$.
 Then the Poisson formula gives (see \cite[Theorem 1]{CB}) $$h_3(x)=\int_{|y|>4}P(x,y)v^+(y)dy,\quad x\in B_4,$$
 where $$P(x,y)=C_n\left(\frac{16-|x|^2}{|y|^2-16}\right)^{\frac{1}{2}}\frac{1}{|x-y|^n}.$$
 Now for $x\in B_2$ by H\"older's inequality we get
 \begin{align}
  |h_3(x)| &\leq C\int_{|y|>4}\left(\frac{1}{|y|^2-16}\right)^{\frac{1}{2}}\frac{1}{|y|^n}v^+(y)dy \notag \\
  &\leq C\left(\int_{|y|>4}v^+(y)^3dy\right)^{\frac 13}\left(\int_{|y|>4}\frac{1}{(|y|^2-16)^{\frac 34}}\frac{1}{|y|^{\frac{3n}{2}}}dy\right)^{\frac 23}\notag\\
  &\leq C\|v^+\|_{L^3(\rn)}\leq C,\notag
 \end{align}
where the last inequality follows from \eqref{v-Lp}.
By Lemma \ref{1} below we have $$h\leq C,\text{ for every }x\in B_1(x_0),$$ where $C$ being independent of $x_0$.
\end{proof}

\begin{lem}\label{1}
 Let $h\in C^0(\rn)$ solves \eqref{def-h}. Let $h_1\in C^0(\rn)$ be the solution of
$$
\left\{
 \begin{array}{ll}
 \displaystyle (-\Delta)^{\frac{1}{2}}h_1 =(-\Delta)^{\frac{1}{2}}h & in\, B_4(x_0)  \\
 h_1=0 & on \,B_4(x_0)^c.
 \end{array}
\right.
$$
Then there exists a constant $C=C(n)$ such that $$\|h_1\|_{L^{\infty}(B_1(x_0))}\leq C.$$
\end{lem}
\begin{proof}
 We assume that $x_0=0$. Using Green's representation formula (see \cite[Theorem 3]{CB}) the solution is given by 
 $$h_1(x)=\int_{B_4}G_2(x,y)(-\Delta)^{\frac{1}{2}}h(y)dy, \quad x\in B_4,$$ 
 where  $$G_2(x,y)=C_n|x-y|^{1-n}\int_0^{r_0(x,y)}\frac{r^{\frac{1}{2}-1}}{(1+r)^{\frac{n}{2}}}dr,\quad r_0(x,y)=\frac{(16-|x|^2)(16-|y|^2)}{|x-y|^2}.$$
 Since $$\frac{r^{-\frac{1}{2}}}{(1+r)^{\frac{n}{2}}}\in L^1((0,\infty)),$$ we have $$|G_2(x,y)|\leq C|x-y|^{1-n}.$$
 For $|z|\leq 1$ using \eqref{def-h}, Lemma \ref{derivative-v}  and Lemma \ref{Pizzetti_2} below  we bound
 \begin{align}
  |h_1(z)|&\leq \int_{B_4}|G_2(z,y)||(-\Delta)^{\frac{1}{2}}h(y)|dy \notag\\
  &\leq \sum_{i=0}^{\frac{n-3}{2}}\int_{B_4}|G_2(z,y)|\left(\int_{\partial B_4}\left|(-\Delta)^i(-\Delta)^{\frac{1}{2}}v(x)\right|\left|\frac{\partial}{\partial\nu}
        \left((-\Delta)^{\frac{n-3}{2}-i}G(y,x)\right)\right|d\sigma(x)\right)dy \notag\\
 &\leq C\sum_{i=0}^{\frac{n-3}{2}}\int_{B_4}|z-y|^{1-n}\left(\int_{\partial B_4}\left|(-\Delta)^i(-\Delta)^{\frac{1}{2}}v(x)\right|
       \left|x-y\right|^{1+2i-n} d\sigma(x)\right)dy \notag\\
 &= C\sum_{i=0}^{\frac{n-3}{2}}\int_{|x|=4}\left|(-\Delta)^i(-\Delta)^{\frac{1}{2}}v(x)\right|\left(\int_{|y|<4}|z-y|^{1-n}
       \left|x-y\right|^{1+2i-n}dy\right) d\sigma(x)     \notag\\
 &\leq C \sum_{i=0}^{\frac{n-3}{2}}\int_{|x|=4}\left|(-\Delta)^i(-\Delta)^{\frac{1}{2}}v(x)\right|d\sigma(x) \notag\\
& \leq C. \notag
 \end{align}
\end{proof}

\subsection{Proof of Theorem \ref{thm-1b}}
One can verify easily that $(i)\Rightarrow (ii)$-$(vi)$. On the other hand, by Theorem \ref{thm-1} $(ii)$ to $(iv)$ are equivalent. Moreover, 
$(ii)\Rightarrow(i)$ thanks to \cite[Theorem 4.1]{Xu}. To show that $(v)\Rightarrow (i)$ and $(vi)\Rightarrow(i)$ one can follow the arguments 
in \cite{LM}.

Finally to prove \eqref{C} we use \cite[Theorem 6 and Lemma 3]{LM}. Since the polynomial $P$ is bounded from above,  $\deg(P)$ must be even and 
 let it be $2k$. Then $\Delta^kP=C_{0}$ on $\rn$ and $\Delta^{k+1}P=0$ on $\rn$. By \cite[Lemma 3]{LM} we have
 $$\sum_{i=0}^kc_iR^{2i}\Delta^iP(0)=\frac{1}{|B_R|}\int_{B_R}P(x)dx\leq \sup_{\rn} P\leq C, \,\text{ for every }R>0,$$
 where the constants $c_i's$  are positive and hence $C_{0}=\Delta^kP(0)\leq 0$.  We claim that 
 $C_{0}<0$.  Otherwise,  by Theorem \ref{thm-1} and \cite[Theorem 6]{LM}  one gets  $\deg(P)\leq 2k-2$,
 which is a contradiction. 

\appendix
\section{Appendix}
Combining  \cite[Proposition 2.4]{LS} and \cite[Lemma 3.2]{Valdinoci} we state the following proposition:
\begin{prop}\label{value}
 Let $\Omega$ be an open set in $\rn$. Let $u\in C^{2\sigma+\epsilon}(\Omega)\cap L_{\sigma}(\rn)$ for some $\sigma\in (0,1)$ and $\epsilon>0$. Then $(-\Delta)^{\sigma}u$ is continuous in $\Omega$ and
for every $x\in\Omega$   we have
\begin{align}
(-\Delta)^{\sigma}u(x) &=C_{n,\sigma}P.V.\int_{\rn}\frac{u(x)-u(y)}{|x-y|^{n+2\sigma}}dy \notag\\
 &= -\frac{1}{2}C_{n,\sigma}P.V.\int_{\rn}\frac{u(x+y)+u(x-y)-2u(x)}{|y|^{n+2\sigma}}dy, \label{b}
 \end{align}
 where $C^{2\sigma+\epsilon}(\Omega):=C^{0,2\sigma+\epsilon}(\Omega)$ for $2\sigma+\epsilon\leq1$ and 
$C^{2\sigma+\epsilon}(\Omega)=C^{1,2\sigma+\epsilon-1}(\Omega)$ for $2\sigma+\epsilon>1$ and the constant $C_{n,\sigma}$ is given by 
$$C_{n,\sigma}:=\left(\int_{\rn}\frac{1-\cos x_1}{|x|^{n+2\sigma}}dx\right)^{-1}.$$ 
\end{prop}
The advantage of \eqref{b} is that the integral is not singular at the origin  for a $C^2$ function. 

Proof of the following lemma can be found in \cite{A-H}.
 \begin{lem}[Fundamental solution]\label{funda} For $n\ge 3$ odd integer the function 
 $$\Phi(x):=\frac{(\frac{n-3}{2})!}{2\pi^{\frac{n+1}{2}}}\frac{1}{|x|^{n-1}}=\frac{1}{\gamma_n}(-\Delta)^{\frac{n-1}{2}}\log\frac{1}{|x|}$$ is a fundamental solution of $(-\Delta)^{\frac{1}{2}}$ in $\rn$ in 
 the sense that for all 
 $f\in L^1(\rn)$ we have  $\Phi\ast f\in L_{\frac{1}{2}}(\rn)$ and for all $\varphi\in\s(\rn)$
 \begin{align}
  \int_{\rn}(-\Delta)^{\frac{1}{2}}(\Phi\ast f)\varphi dx:=\int_{\rn}(\Phi\ast f)(-\Delta)^{\frac{1}{2}}\varphi dx=\int_{\rn}f\varphi dx.\notag
 \end{align}
\end{lem}

  \begin{lem}\label{comm-3}
  Let $\ell$ be a nonnegative integer.
  Let $v$ be a smooth function on $\rn$ such that 
   $D^{\alpha}v\in L_{\frac{1}{2}}(\rn)$ for every multi-index $\alpha$ with $|\alpha|\leq \ell$. Then 
  $$(-\Delta)^{\frac{1}{2}}D^{\alpha}v(x)=D^{\alpha}(-\Delta)^{\frac{1}{2}}v(x), \qquad \text{for every } x\in\rn , \, |\alpha|\leq \ell. $$ 
 \end{lem}
\begin{proof}
It suffices to show the case for $|\alpha|=1$.
Let $\varphi\in C_c^{\infty}(B_2)$ be such that $\varphi=1$ on $B_1$ and $0\leq\varphi\leq1$. Let us define $v_k(x):=\varphi(\frac{x}{k})v(x).$
Then we have 
\begin{align}
 (-\Delta)^{\frac{1}{2}}D^{\alpha}v_k(x)=D^{\alpha}(-\Delta)^{\frac{1}{2}}v_k(x).\label{smooth}
\end{align}
 We claim that 
 $$(-\Delta)^{\frac{1}{2}}D^{\alpha}v_k\xrightarrow{k\to\infty} (-\Delta)^{\frac{1}{2}}D^{\alpha}v\quad \text{in } C^0_{loc}(\rn),\quad |\alpha|=0,1.$$
To prove our claim first we fix a $R>0$. Then for $x\in B_R$ and $k\geq R+1$ we get 
  \begin{align}
 \left|(-\Delta)^{\frac{1}{2}}D^{\alpha}v_k(x)-(-\Delta)^{\frac{1}{2}}D^{\alpha}v(x)\right|
 &=C_{n,\frac 12}\left|P.V.\int_{\rn}\frac{D^{\alpha}v_k(x)-D^{\alpha}v_k(y)-D^{\alpha}v(x)+D^{\alpha}v(y)}{|x-y|^{n+1}}dy\right|\notag\\
  &\leq C_{n,\frac 12}\int_{|y|>k}\frac{2|D^{\alpha}v(y)|+|\alpha|k^{-1}|(D^{\alpha}\varphi)(\frac yk)||v(y)|}{|x-y|^{n+1}}dy \notag\\
  & \xrightarrow{k\to\infty} 0. \notag
\end{align}
Thus $\{D^{\alpha}(-\Delta)^{\frac{1}{2}}v_k\}_{k=1}^\infty=\{(-\Delta)^{\frac{1}{2}}D^{\alpha}v_k\}_{k=1}^\infty$
and $\{(-\Delta)^{\frac{1}{2}}v_k\}_{k=1}^\infty$ are  Cauchy sequences in $C^0_{loc}(\rn)$ and consequently
 \begin{align}
       D^{\alpha}(-\Delta)^{\frac{1}{2}}v_k(x)\xrightarrow{k\to\infty}D^{\alpha}(-\Delta)^{\frac{1}{2}}v(x),\notag
 \end{align}
and together with \eqref{smooth} complete the proof.
\end{proof}

\begin{lem}\label{Pizzetti_2}
 Let $h\in C^{n-1}(\bar{B_r})$ be such that 
\begin{align}\label{dir}
 \left\{\begin{array}{ll}
   (-\Delta)^{\frac{n-1}{2}}h=0 & in \, B_r\\
   (-\Delta)^jh=f_j & on \, \partial B_r,\, j=0,1,...,\frac{n-3}{2}.\\
 \end{array}
\right.
\end{align}
 Then  for every $x\in B_{r}$
 \begin{align}
 h(x)=-\sum_{i=0}^{\frac{n-3}{2}}\int_{\partial B_r}f_i(y)\frac{\partial}{\partial\nu}\left((-\Delta)^{\frac{n-3}{2}-i}G(x,y)\right)d\sigma(y),\notag 
 \end{align}
 and 
\begin{align}\label{piz}
 |h(x)|\leq C\sum_{i=0}^{\frac{n-3}{2}}\int_{\partial B_r}|f_i(y)|\frac{1}{|x-y|^{n-1-2i}}d\sigma(y),
\end{align}
 where 
 $G$ is the Green's function corresponding to the problem \eqref{dir}.
\end{lem}
\begin{proof}
 Using integration by parts we have
 \begin{align}
  0&=\int_{B_r}G(x,y)(-\Delta)^{\frac{n-1}{2}}h(y)dy \notag\\
  &=\sum_{i=0}^{\frac{n-3}{2}}\int_{\partial B_r}(-\Delta)^ih(y)\frac{\partial}{\partial\nu}\left((-\Delta)^{\frac{n-3}{2}-i}G(x,y)\right)d\sigma(y)
           +\int_{B_r}(-\Delta)^{\frac{n-1}{2}}G(x,y)h(y)dy\notag       \\
  &= h(x)+\sum_{i=0}^{\frac{n-3}{2}}\int_{\partial B_r}f_i(y)\frac{\partial}{\partial\nu}\left((-\Delta)^{\frac{n-3}{2}-i}G(x,y)\right)d\sigma(y)\notag
 \end{align}
 To get \eqref{piz} we only need to show that 
 $$\left|\frac{\partial}{\partial y_i}(-\Delta)^jG(x,y)\right|\leq \frac{1}{|x-y|^{2+2j}},\quad x,y\in B_r,\,0\leq j\leq\frac{n-3}{2}.$$
 In order to do that we use the following representation formula of $G$ given by (see e.g. \cite{Grunau})
  $$G(x,y)=\underbrace{\int_{B_r}\dots\int_{B_r}}_{\text{$\frac{n-3}{2}$ times}}G_1(x,z_1)G_1(z_1,z_2)\dots G_1(z_{\frac{n-3}{2}},y)
 dz_1dz_2\dots dz_{\frac{n-3}{2}},\quad x,y\in B_r,$$
where 
$$G_1(x,y)=\frac{1}{n(n-2)|B_1|}\left(\frac{1}{|x-y|^{n-2}}-\frac{r^{n-2}}{||x|(y-\frac{r^2x}{|x|^2})|^{n-2}}\right)\quad x,y\in B_r,$$
is the Green's function for Laplacian on $B_r$. Then  for $0\leq j\leq\frac{n-3}{2}$ 
$$(-\Delta)^jG(x,y)=\underbrace{\int_{B_r}\dots\int_{B_r}}_{\text{$\frac{n-3-2j}{2}$ times}}G_1(x,z_1)G_1(z_1,z_2)\dots G_1(z_{\frac{n-3-2j}{2}},y)dz_1dz_2
\dots dz_{\frac{n-3-2j}{2}},$$
and 
$$\frac{\partial}{\partial y_i}(-\Delta)^jG(x,y)=\underbrace{\int_{B_r}\dots\int_{B_r}}_{\text{$\frac{n-3-2j}{2}$ times}}G_1(x,z_1)G_1(z_1,z_2)
\dots \frac{\partial}{\partial y_i}G_1(z_{\frac{n-3-2j}{2}},y)dz_1dz_2\dots dz_{\frac{n-3-2j}{2}}.$$
A repeated use  of Lemma \ref{p3} and the estimate 
$$0<G_1(x,y)\leq \frac{C}{|x-y|^{n-2}}\text{ and } \left|\frac{\partial}{\partial x_i}G_1(x,y)\right|\leq \frac{C}{|x-y|^{n-1}}\quad x,y\in B_r,$$
gives
$$\left|\frac{\partial}{\partial y_i}(-\Delta)^jG(x,y)\right|\leq C\int_{B_r}\frac{1}{|x-z|^{3+2j}}\frac{1}{|y-z|^{n-1}}dz
\leq C\frac{1}{|x-y|^{2+2j}},\quad 0\leq j\leq\frac{n-3}{2}.$$ 
\end{proof}

\begin{lem}\label{log} We set
$$f_0(x):=\log|x|,\quad f_j(x):=\frac{1}{|x|^{j}} \text{ for }j=1,2,\dots,n-1.$$
Then for $0<\sigma<1$ we have
 $$(-\Delta)^{\sigma} f_j(x)=\frac{1}{|x|^{j+2\sigma}}(-\Delta)^{\sigma} f_j(e_1),\quad \text{ for }|x|>0\text{ and }0\le j\le n-1.$$ 
\end{lem}
\begin{proof} Since $f_j\in C^{\infty}(\rn\setminus\{0\})\cap L_{\frac{1}{2}}(\rn)$ using \eqref{b} we get
 \begin{align}
  (-\Delta)^{\sigma} f_j(x)&=(-\Delta)^{\sigma} f_j(|x|e_1)=c_nP.V. \int_{\rn}\frac{f_j(|x|e_1)-f_j(y)}{\left||x|e_1-y\right|^{n+2\sigma}}dy \notag\\
  &=\frac{1}{|x|^{j+2\sigma}}c_nP.V. \int_{\rn}\frac{f_j(e_1)-f_j(y)}{|e_1-y|^{n+2\sigma}}dy\notag\\
  &=\frac{1}{|x|^{j+2\sigma}}(-\Delta)^{\sigma} f_j(e_1),\notag
   \end{align}
where in the first equality we used that the function  $(-\Delta)^{\sigma} f_j$ is radially symmetric.
\end{proof}

The following lemma is a variant of \cite[Theorem 6]{LM}.
\begin{lem}\label{deg-pol}
 Let $v\in L_\frac n2(\rn)$ and let $h=u-v$ be  $\frac{n+1}{2}$-harmonic in $\rn$ i.e. $$\Delta ^\frac{n+1}{2}h=0,\quad \text{in }\rn.$$ If $u$ satisfies \eqref{deg} then 
 $h$ is a polynomial of degree at most $n-1$.
\end{lem}
\begin{proof}
First notice that the condition  $v\in L_\frac n2(\rn)$ implies that $$\int_{B_R}|v|dx=o(R^{2n})\quad \text{as }R\to\infty.$$
 For a fixed $x\in\rn$  by \cite[Proposition 4]{LM} we have 
 $$|D^\alpha h(x)|\leq \frac{C}{R^{2n}}\int_{B_R(x)}|h(y)|dy\leq \frac{C}{R^{2n}}\int_{B_{2R}}|h(y)|dy,\quad \alpha\in\mathbb{N}^n\text{ with }|\alpha|=n,\text{ as } R\to\infty.$$
Now using \eqref{deg} $$\int_{B_R}h^+dx\leq \int_{B_R}(u^++|v|)dx=o(R^{2n})\quad\text{ or }\int_{B_R}h^-dx\leq \int_{B_R}(u^-+|v|)dx=o(R^{2n}).$$
On the other hand, Pizzetti's formula (see e.g. \cite[Lemma 3]{LM}) implies that 
 $$\int_{B_R}hdx=O(R^{2n-1}), \quad \text{as }R\to\infty.$$ Therefore,
 \begin{align}
  |D^\alpha h(x)|\leq \frac{C}{R^{2n}}\min\left\{\int_{B_{2R}}(2h^+-h)dy,\int_{B_{2R}}(2h^-+h)dy\right\}&=\frac{1}{R^{2n}}\left(O(R^{2n-1})+o(R^{2n})\right)\notag\\
  &\xrightarrow{R\to\infty}0,\notag
 \end{align}
  and hence $h$ is a polynomial of degree at most $n-1$.
\end{proof}

\medskip

\noindent\textbf{Acknowledgements} I would like to thank my advisor Prof. Luca Martinazzi for suggesting the problem and for many stimulating conversations.
I would also like to thank Prof. Adimurthi for a useful discussion.


\end{document}